\documentclass{article}

\usepackage{amsrefs,amssymb,amsmath,amsthm,amsfonts,epsfig,graphicx}
\usepackage[utf8]{inputenc}
\usepackage[T1]{fontenc}
\usepackage[english]{babel}
\usepackage{hyperref}

\hypersetup{
    colorlinks,
    citecolor=black,
    filecolor=black,
    linkcolor=black,
    urlcolor=black
}

\theoremstyle{plain}
\newtheorem{teo}{Theorem}[section]
\newtheorem{cor}[teo]{Corollary}

\newtheorem{prop}[teo]{Proposition}
\newtheorem{rmk}[teo]{Remark}
\theoremstyle{definition}
\newtheorem{df}[teo]{Definition}
\newtheorem{exa}[teo]{Example}

\DeclareMathOperator{\dist}{dist}

\DeclareMathOperator{\card}{card}

\newcommand{\Bigcap}{\textstyle\operatorname*{\textstyle\bigcap}}

\newcommand{\Bigcup}{\textstyle\operatorname*{\textstyle\bigcup}}

\newcommand{\U}{{\mathcal U}}
\newcommand{\V}{{\mathcal V}}

\newcommand{\Z}{\mathbb Z}
\newcommand{\N}{\mathbb N}
\renewcommand{\epsilon}{\varepsilon}

\title{Expansive Homeomorphisms on non-Hausdorff Spaces}

\author{M. Achigar, A. Artigue and I. Monteverde}

\begin{document}

\maketitle

\begin{abstract}
  We study expansive dynamical systems from the viewpoint of general topology. We introduce the notions of orbit and refinement expansivity on topological spaces extending expansivity in the compact metric setting. Examples are given on non-Hausdorff compact spaces. Topological properties are studied in relation to separability axioms, metrizability and uniform expansivity.
\end{abstract}

\section{Introduction}

 \par\noindent\par Given a compact metric space $(X,\dist)$ a homeomorphism $f\colon X\to X$ is \emph{expansive} \cite{Utz} (or \emph{metric expansive}) if there is $\delta>0$ such that if $\dist(f^n(x),f^n(y))<\delta$ for all $n\in\Z$ then $x=y$. This definition was stated from a topological viewpoint, i.e. without mentioning the metric, in for example \cites{Br60,Fried,KR}. In these papers their hypothesis allowed them to prove the metrizability of the space, and consequently the concept they consider is in fact metric expansivity. The purpose of the present paper is to investigate topological definitions of expansivity with examples on non-metrizable topologies and to extend known result of metric expansivity to this topological setting.

 \par We introduce two of such definitions. In order to explain the first one let us consider a metric expansive homeomorphism $f$ as above. Since $X$ is compact, there are $x_1,\dots,x_l\in X$ such that $X=\Bigcup_{i=1}^l B_{\delta/2}(x_i)$, where $B_r(p)$ is the open ball of radius $r$ centered at $p$. If $U_i=B_{\delta/2}(x_i)$ for $i=1,\dots,l$ and $f$ is expansive we have constructed an open cover $U_1,\dots,U_l$ satisfying: if $x\neq y$ then there is $n\in\Z$ such that $\{f^n(x),f^n(y)\}\nsubseteq U_i$ for all $i=1,\dots,l$. We use this property of expansive homeomorphisms to extend the definition for non-metric topological spaces. If a covering $U_1,\dots,U_l$ as above exists we will say that $f$ is \emph{orbit expansive}. A similar idea was previously considered in \cite{KR} but assuming that the space is Hausdorff. In this case, as we will show \textcolor{black}{in Theorem \ref{teoMetrico}}, the orbit expansivity coincides with the metric expansivity. We will give examples of orbit expansive homeomorphisms on non-Hausdorff compact spaces, see Examples \ref{exaNonHausdorff} and \ref{exadupli}.

 \par Let us explain the second definition considering again a metric expansive homeomorphism $f$. Consider the open cover of small balls $\mathcal U=\{U_1,\dots,U_l\}$ defined above. Then we can prove that for every open cover $\mathcal V$ and every sequence $(k_j)_{j\in\Z}\in\{1,\dots,l\}^{\Z}$ there exists an $N\in\N$ such that $\bigcap_{|j|\leq N}f^j(U_{k_j})$ is contained in some open set of $\mathcal V$. In the topological setting this definition will be called \emph{refinement expansivity}. This concept is interesting because it does not mention the metric neither points, it is stated in terms of the dynamics of open covers of the space. This allows us to give examples on non-$T_1$ spaces, see Examples \ref{exaNoT0} and \ref{exaNoT1}. But, assuming that the space is $T_1$ we prove in Theorem \ref{teoRefT1Cov} that refinement expansivity implies orbit expansivity. As in the metric case \cite{Br62} we prove, see Theorem \ref{ref=>unifref}, a uniform version of this expansivity.

\vspace{3mm}

 \par The paper is organized as follows. In Section \ref{secOexp} we introduce orbit expansivity. We relate our definition with the concept of generators of \cite{KR}. Also, a definition via an isolating neighborhood of the diagonal of $X\times X$ is considered. We show that if a space $X$ admits an orbit expansive homeomorphism then $X$ is $T_1$. In Theorem \ref{teoMetrico} we show that if a compact Hausdorff space admits an orbit expansive homeomorphism then the space is metrizable. As a corollary we obtain that the compact unit square with the lexicographic topology does not admit orbit expansive homeomorphisms. In Example \ref{exaNonHausdorff} we give an example of a compact non-Hausdorff space admitting an orbit expansive homeomorphism. We also prove that no infinite set with the cofinite topology admits orbit expansive homeomorphisms. Other properties are proved in relation with the cardinality of the set of periodic points and the expansivity of the powers of the orbit expansive homeomorphisms.

 \par In Section \ref{secRefExp} we introduce the notion of refinement expansivity. We start giving examples showing that it is natural to assume that our space is $T_1$. In Section \ref{secUnifExp} we introduce uniform refinement expansivity and we prove that this definition coincides with refinement expansivity. This allows us to conclude that if a space $X$ admits a refinement expansive homeomorphism then $X$ is compact, which is an important difference with orbit expansivity. Then, we study the restriction to invariant subspaces via the extension closed property from \cite{Harris}. In Section \ref{secRefImpOrb} we prove that on $T_1$ spaces refinement expansivity implies orbit expansivity. In Example \ref{exadupli} we show that a $T_1$ but non-Hausdorff topological space may admit a refinement expansive homeomorphism.

\vspace{3mm}

 \par We thank Damián Ferraro, Armando Treibich and José Vieitez for useful conversations on these topics.

\section{Orbit expansivity}\label{secOexp}

 \par\noindent\par Let $(X,\tau)$ be a topological space and consider $f\colon X\to X$ a homeomorphism. We do not assume compacity of $X$. For a set $A\subseteq X$ and a cover $\mathcal C$ of $X$ we write $A\prec\mathcal C$ if there exists $C\in\mathcal C$ such that $A\subseteq C$. If $\mathcal A$ is a family of subsets of $X$ then $\mathcal A\prec \mathcal C$ means that $A\prec\mathcal C$ for all $A\in \mathcal A$. If in addition $\mathcal A$ and $\mathcal C$ are coverings of $X$ we say that $\mathcal A$ is a \emph{refinement} of $\mathcal C$.

\begin{df}
  We say that $f$ is \emph{orbit expansive} if there is a finite open cover $\mathcal U=\{U_1,\dots,U_l\}$ of $X$ such that if $\{f^n(x),f^n(y)\}\prec\mathcal U$ for all $n\in\Z$ then $x=y$. In this case we call $\mathcal U$ an \emph{o-expansive covering}.
\end{df}

 \par The notation suggests to think of $\{f^n(x),f^n(y)\}\prec\mathcal U$ as $$\dist(f^n(x),f^n(y))<\delta$$ in the metric case.

\begin{rmk}\label{dfexpcard}
  It is easy to prove that $\mathcal U$ is an o-expansive covering if and only if
    $$\card\Bigl(\bigcap_{j\in\Z}f^j(U_{k_j})\Bigr)\leq1\text{ for all }(k_j)_{j\in\Z}\in\{1,\dots,l\}^{\Z},$$
  where $\card$ stands for the cardinality of the set. In \cite{KR} a related concept is considered that they call \emph{generator}.
\end{rmk}

\begin{rmk}
  Notice that if the empty set $\emptyset$ is in an o-expansive covering $\mathcal U$ then $\mathcal U\setminus\{\emptyset\}$ is an o-expansive covering too. It is also easy to see that if $X\in\mathcal U$ then $\card(X)=1$.
\end{rmk}

 \par As usual, we define the \emph{diagonal} of $X\times X$ as 
   $$\Delta=\{(x,x):x\in X\}.$$

\begin{prop}\label{propDiagAislada} 
  Let $f\colon X\to X$ be a homeomorphism on a compact space $X$. The following statements are equivalent: 
  \begin{enumerate}
    \item $f$ is orbit expansive,
    \item there is a neighborhood $U\subset X\times X$ of the diagonal such that if $x\neq y$ then there is $n\in\Z$ such that $(f^n(x),f^n(y))\notin U$.
  \end{enumerate}
\end{prop}

\begin{proof}
  (1 $\to$ 2). Given an o-expansive covering $U_1,\dots U_l$, consider $U=U_1\times U_1\cup\dots\cup U_l\times U_l$. It is easy to see that $U$ satisfies condition 2. 

  \par (2 $\to$ 1). Given $U$ consider for each $x\in X$ a neighborhood $U_x$ of $x$ such that $U_x\times U_x\subset U$. Since $X$ is compact, we know that the diagonal $\Delta$ is compact. Then there is a finite cover, say $U_{x_1}\times U_{x_1},\dots,U_{x_l}\times U_{x_l}$ for some finite set $\{x_1,\dots,x_l\}\subset X$. Then, taking $U_i=U_{x_i}$, for $i=1,\dots,l$, we obtain an o-expansive covering $\{U_1,\dots,U_l\}$.
\end{proof}

 \par Notice that implication (1 $\to$ 2) in the previous Proposition does not need the compacity of $X$. As usual, we say that $X$ is $T_1$ if given $x,y\in X$, $x\neq y$, then there are two open sets $U,V$ such that $x\in U$, $y\in V$, $x\notin V$ and $y\notin U$.

\begin{prop}\label{propCovExpT1}
  If $X$ admits an orbit expansive homeomorphism then $X$ is a $T_1$ space.
\end{prop}

\begin{proof}
  Given distinct points $x$ and $y$ of $X$ take $n\in\Z$ such that $\{f^n(x),f^n(y)\}\nprec\mathcal U$ where $\mathcal U$ is an o-expansive covering. As $f^n(x)\in U$ for some $U\in\mathcal U$ we have $y\notin f^{-n}(U)$, $x\in f^{-n}(U)$. This proves that $X$ is $T_1$.
\end{proof}

\begin{rmk}
  If $X$ is finite and it admits an orbit expansive homeomorphism then $X$ has the discrete topology. 
  This follows by Proposition \ref{propCovExpT1}.
\end{rmk}

 \par Orbit expansivity is equivalent to the usual notion of expansivity if the space is Hausdorff, as the next result shows. Similar results, but with different definitions, were obtained in \cites{Br60,Fried,KR}. The closure of a set $V$ will be denoted as $\overline V$.

\begin{teo}\label{teoMetrico}
  If a compact Hausdorff topological space admits an orbit expansive homeomorphism then it is metrizable. 
\end{teo}

\begin{proof}
  Let $\mathcal U$ be an o-expansive covering. As $X$ is compact and Hausdorff we can take an open cover $\mathcal V=\{V_1,\dots,V_m\}$ such that $\{\overline V:V\in\mathcal V\}\prec\mathcal U$ and satifying
    $$\card\Bigl(\bigcap_{j\in\Z}f^j(\overline{V_{k_j}})\Bigr)\leq1\text{ for all }(k_j)_{j\in\Z}\in\{1,\dots,m\}^{\Z}.$$
  Then, given a point $x$ and a neighborhood $W$ of $x$, by a compacity argument, we can find an $N\in\N$ and a sequence $(k_j)_{j=-N}^N$ such that
    $$x\in\bigcap_{|j|\leq N}f^j(V_{k_j})\subseteq W.$$
  So the finite intersections of the elements of $\mathcal V$ and its iterates form a countable basis for the topology of $X$. Therefore, being compact and Hausdorff, we can apply Urysohn's Theorem (see for example \cite{HY}) to conclude that $X$ must be metrizable. 
\end{proof}

 \par The following remark is a simple corollary of Theorem \ref{teoMetrico}.

\begin{rmk}
  The unit square $X=[0,1]\times[0,1]$ with the topology induced by the lexicographic order does not admit orbit expansive homeomorphisms. This is because $X$ is compact and Hausdorff but it is not metrizable.
\end{rmk}

\begin{rmk}
  If there is a closed set $U\subset X\times X$ such that 
  $\Bigcap_{i\in\Z} (f\times f)^i(U)=\Delta$ then $\Delta$ is a closed set. This implies that $X$ is Hausdorff. 
  So, if the neighborhood $U$ of the diagonal of Proposition \ref{propDiagAislada} can be taken as a closed neighborhood then we conclude by Theorem \ref{teoMetrico} that $X$ is metric. 
  Consequently, the topological expansivity defined in \cite{DLRW} is metric expansivity if the space is compact.
\end{rmk}

\begin{rmk}
  If $X$ is a compact and Hausdorff space and $f$ is orbit expansive then for every metric $\dist$ defining the topology of $X$ there is $\delta>0$ such that if $x,y\in X$ and $x\neq y$ then there is $n\in\Z$ such that $\dist(f^n(x),f^n(y))\geq\delta$. In fact, given a metric $\dist$ defining the topology of $X$ 
  it suffices to take $\delta$ as the Lebesgue number of an o-expansive covering of $f$.
\end{rmk}

 \par Non-compact spaces may admit orbit expansive homeomorphisms, see the next example. In the following section we show that this is not the case for refinement expansivity.

\begin{exa}
\label{exaNoCompact}
  Consider the metric space $\Z$ with its usual metric, i.e., the discrete topology.  
  Define the homeomorphism $f\colon \Z\to \Z$ as $f(n)=n+1$ and the open sets
  $U_1=\{n\in\Z:n\geq 0\}$ and $U_2=\{n\in\Z:n<0\}$. 
  It is easy to see that $\mathcal U=\{U_1,U_2\}$ is an o-expansive cover for $f$.
\end{exa}

 \par Let us give an example of an orbit expansive homeomorphism on a non-Hausdorff space.

\begin{exa}
\label{exaNonHausdorff}
  Consider $f\colon\Z\to\Z$ the homeomorphism of Example \ref{exaNoCompact}. Define $X$ as a non-Hausdorff compactification of $\Z$ with two points $\infty_1$ and $\infty_2$ such that a basis of neighborhoods of $\infty_i$ is formed by the sets $U$, $\infty_i\in U$, with finite complement. We have that $f$ can be extended to a homeomorphism of $X$ by fixing each $\infty_i$. In this way we obtain an orbit expansive homeomorphism on a compact non-Hausdorff space. More details are given in Example \ref{exadupli}, where this example is generalized.
\end{exa}

 \par A well known example of a compact topological space that is $T_1$ but not Hausdorff is the cofinite topology on an infinite set. We show that such space admits no orbit expansive homeomorphisms.

\begin{prop}
  If $X$ is an infinite set with the cofinite topology then no homeomorphism of $X$ is orbit expansive.
\end{prop}

\begin{proof}
  Arguing by contradiction, suppose there exists an o-expansive covering $\{U_1,\dots,U_l\}$ for a homeomorphism $f\colon X\to X$. Consider the finite sets $A_i=X\setminus U_i$ and define $A=\Bigcup A_i$. If every point is periodic then there is an infinite set of points that never visit $A$, in contradiction with the orbit expansivity. Then, we can take a non-periodic point $x\in X$. Since $X$ is infinite, there is $N\geq 1$ such that $f^i(x)\notin A$ if $|i|\geq N$. The points $x$ and $y=f^{2N}(x)$ contradict the orbit expansivity because, for each $i\in\Z$ at most one of the points $f^i(x)$ and $f^i(y)$ is in $A$. This finishes the proof.
\end{proof}

 \par The following results extends well known properties of metric expansivity, see \cite{Utz}. They should be compared with what is obtained in the next section.

\begin{prop}
  Let $f\colon X\to X$ be a homeomorphism and $r\neq 0$ an integer. Then, $f$ is orbit expansive if and only if $f^r$ is.
\end{prop}

\begin{proof}
  If $\mathcal U=\{U_1,\dots,U_n\}$ is an o-expansive covering of $f$ then $\bigl\{U_{k_1}\cap f(U_{k_2})\cap\cdots\cap f^{r-1}(U_{k_{r-1}}): k_1,k_2,\ldots,k_{r-1}\in\{1,\dots,n\}\bigr\}$ is an o-expansive covering for $f^r$. From this the result follows.
\end{proof}

\begin{prop}
  If $\mathcal U=\{U_1,\dots,U_l\}$ is an o-expansive covering for $f\colon X\to X$ then for each $n\geq 1$ there are at most $l^n$ points $p\in X$ such that $f^n(p)=p$. In particular, the set of periodic points is countable.
\end{prop}

\begin{proof}
  It is an easy application of the Pigeonhole Principle.
\end{proof}

\begin{prop}\label{propRestOexp}
  If $f\colon X\to X$ is orbit expansive, $Y\subseteq X$ and $f(Y)=Y$ then $f\colon Y\to Y$ is orbit expansive too.
\end{prop}

\begin{proof}
  The restriction to $Y$ of an o-expansive covering gives an o-expansive covering for $f\colon Y\to Y$.
\end{proof}

\section{Refinement expansivity}\label{secRefExp}

 \par\noindent\par In this section we consider another kind of topological expansivity. As before, let us consider a homeomorphism $f\colon X\to X$ of a topological space $X$. We do not assume the compacity of $X$.

\begin{df}
  We say that $f$ is \emph{refinement expansive} if there exists an open cover $\mathcal U=\{U_1,\dots,U_n\}$ of $X$, called an \emph{r-expansivity covering}, such that for every open cover $\mathcal V$ and $(k_j)_{j\in\Z}\in\{1,\dots,n\}^{\Z}$ there exists an $N\in\N$ such that $\bigcap_{|j|\leq N}f^j(U_{k_j})\prec\mathcal V$.
\end{df}

 \par The purpose of the following examples and remarks is to show that it is desirable to assume that $X$ is $T_1$.

\begin{exa}\label{exaNoT0}
  Let us give a trivial example. Consider $X$ endowed with the topology $\tau=\{X,\emptyset\}$ and $f\colon X\to X$ an arbitrary homeomorphism (i.e., bijection). Trivially, $f$ is refinement expansive with r-expansivity covering $\mathcal U=\{X\}$. By Proposition \ref{propCovExpT1} we have that if $\card(X)>1$ then $f$ is not orbit expansive. 
\end{exa}

\begin{rmk}
  Given a topological space $(X,\tau)$ consider the equivalence relation given by $x\sim y$ if for all $U\in\tau$ it holds that $x,y\in U$ or $\{x,y\}\cap U=\emptyset$. The quotient space $\tilde X=X/\sim$ is a $T_0$ space. Every homeomorphism $f\colon X\to X$ induces a quotient homeomorphism $\tilde f$ of $\tilde X$. It is easy to see that $f$ is refinement expansive if and only if $\tilde f$ is refinement expansive. This means, we can assume that $X$ is $T_0$.
\end{rmk}

 \par However, even assuming that the space is $T_0$ we have \emph{undesired} examples.

\begin{exa}\label{exaNoT1}
  Consider $X=[0,1)$ with the topology $\tau_X=\{(a,1):1>a\geq0\}\cup\{X,\varnothing\}$ and $f\colon X\to X$ given by $f={\rm id}_X$. The space $X$ is compact, $T_0$, not $T_1$, $f$ is refinement expansive and has infinitely many fixed points. Also notice that $f$ is not orbit expansive.
\end{exa}

 \par In light of Example \ref{exaNoT1}, it is natural to assume that the space $X$ is $T_1$. In Theorem \ref{teoRefT1Cov} we will show that every refinement expansive homeomorphism on a $T_1$ space is orbit expansive.

\subsection{Uniform refinement expansivity}\label{secUnifExp}

 \par\noindent\par Consider a homeomorphism $f\colon X\to X$ of a topological space $X$. Given an open cover $\mathcal U=\{U_1,\dots,U_n\}$ of $X$ we define another open cover
   $$\mathcal U_{f,N}=\left\{\Bigcap_{|j|\leq N}f^j(U_{k_j}):(k_j)_{j\in\Z}\in\{1,\dots,n\}^{\Z}\right\}$$
 for any given $N\in\N$.

\begin{df}
  We say that $f$ is \emph{uniformly refinement expansive} if there exists a finite open cover $\mathcal U$ of $X$ such that for every open cover $\mathcal V$ there exist an $N\in\N$ such that $\mathcal U_{f,N}\prec\mathcal V$. In this case we say that $\mathcal U$ is a \emph{uniform r-expansivity covering}. 
\end{df}

\begin{rmk}
  This definition is related with the uniform expansivity in the compact metric case as was introduced in \cite{Br62}.
\end{rmk}

\begin{rmk}
  Clearly, every uniform r-expansivity covering is an r-expansivity covering. Consequently, uniform refinement expansivity implies refinement expansivity. The converse is our next result.
\end{rmk}

\begin{teo}\label{ref=>unifref}
  If $\mathcal U$ is an r-expansivity covering for $(X,f)$ then $\mathcal U$ is a uniform r-expansivity covering. If $f$ is refinement expansive, then $f$ is uniformly refinement expansive.
\end{teo}

\begin{proof}
  Arguing by contradiction assume that $\mathcal U=\{U_1,\dots,U_n\}$ is an r-expansive covering wich is not uniform. Then there exists an open covering $\mathcal V$ of $X$ such that $\U_{f,N}$ does not refine $\mathcal V$ for every $N\in\N$. That is, for all $N\in\N$ there exists $(k^N_j)_{j\in\Z}\in\{1,\dots,n\}^{\Z}$ such that 
  \begin{equation}\label{eqRef}
    \bigcap_{|j|\leq N}f^j(U_{k^N_j})\not\prec\mathcal V. 
  \end{equation}
  We will show that there is a sequence $(k_j)_{j\in\Z}\in\{1,\dots,n\}^{\Z}$ so that 
  \begin{equation}\label{eqRef2}
    \bigcap_{|j|\leq N}f^j(U_{k_j})\not\prec\mathcal V\hbox{ for any }N\in\N, 
  \end{equation}
  wich means that $\U$ si not a r-expansivity covering. This contradiction will prove the theorem. In order to define the sequence $(k_j)_{j\in\Z}$ take $k_0\in\{1,\dots,n\}$ such that the set $\{N\in\N:k^N_0=k_0\}$ is infinite. For $j=\pm 1$ take $k_1,k_{-1}\in\{1,\dots,n\}$ such that the set $\{N\in\N:k^N_j=k_j\hbox{ if } |j|\leq 1\}$ is infinite. Having defined $k_j$ for $|j|\leq j'$, take $k_{j'+1},k_{-j'-1}\in\{1,\dots,n\}$ such that the set $\{N\in\N:k^N_j=k_j\hbox{ if } |j|\leq j'+1\}$ is infinite. Inductively, this defines a sequence $(k_j)_{j\in\Z}$ satisfying
  \begin{equation}\label{ecuCard}
    \card\{M\in\N:k^M_j=k_j\text{ if } |j|\leq N\}=\infty
  \end{equation}
  for all $N\in\N$. Suppose that $\bigcap_{|j|\leq N_0}f^j(U_{k_j})\prec\mathcal V$ for some $N_0\in\N$. Then, by (\ref{ecuCard}), there is an $M\geq N_0$ such that $k^{M}_j=k_j$ for all $|j|\leq N_0$. Therefore, $\bigcap_{|j|\leq N_0}f^j(U_{k^{M}_j})\prec\mathcal V$. This implies that $\bigcap_{|j|\leq M}f^j(U_{k^{M}_j})\prec\mathcal V$ wich in turn contradicts condition (\ref{eqRef}). As we said, this finishes the proof.
\end{proof}

 \par In Example \ref{exaNoCompact} we proved that a non-compact space may admit an orbit expansive homeomorphism. As a consequence of the uniformity we now obtain that this is not the case for refinement expansive homeomorphisms.

\begin{cor}
  If $X$ admits a refinement expansive homeomorphism then $X$ is compact.
\end{cor}

\begin{proof}
  By Theorem \ref{ref=>unifref} we know that $f$ is uniformly refinement expansive. Let $\mathcal U$ be a uniform r-expansivity cover of $X$. Given an arbitrary open cover $\V$ of $X$ take $N\in\N$ such that $\U_{f,N}$ refines $\V$. Then, as $\U_{f,N}$ is a finite cover of $X$, $\V$ has a finite subcover. This proves the compacity of $X$.
\end{proof}

 \par In Proposition \ref{propRestOexp} we showed that the restriction of an orbit expansive homeomorphism to an invariant subset is again orbit expansive. This is not the case for refinement expansivity because such invariant subset may not be compact. This motivates the introduction of the concept of \emph{extension closed} subspaces from \cite{Harris}.

\begin{df}
  A subspace $Y\subset X$ is \emph{extension closed} if every open cover of $Y$ extends to an open cover of $X$. This means, if $\{U_i\}_{i\in I}$ is an open cover of $Y$ (with the subspace topology) then there is an open cover $\{V_i\}_{i\in I}$ of $X$ such that $U_i=V_i\cap Y$.
\end{df}

 \par As noticed in \cite{Harris}, every closed subset is extension closed. Also: 1) on Hausdorff spaces the concept of \emph{extension closed} and \emph{closed} coincides and 2) an extension closed subspace of a compact space is compact.

\begin{prop}
  If $f\colon X\to X$ is refinement expansive and $Y\subset X$ is extension closed and $f(Y)=Y$ 
  then $f\colon Y\to Y$ is refinement expansive.
\end{prop}

\begin{proof}
  If $\mathcal U=\{U_1,\dots,U_l\}$ is an r-expansive covering and $Y\subset X$ is invariant and extension closed, consider $\mathcal U^Y=\{U^Y_1,\dots,U^Y_l\}$ where $U^Y_i=U_i\cap Y$. We will show that $\mathcal U^Y$ is an r-expansive covering for $f\colon Y\to Y$. For this, take a covering $\mathcal V^Y$ of $Y$. Since $Y$ is extension closed, we can take a covering $\mathcal V$ of $X$ extending $\mathcal V^Y$. Now, given $(k_j)_{j\in\Z}\in\{1,\dots,l\}^{\Z}$ we know that there exists an $N\in\N$ such that $\bigcap_{|j|\leq N}f^j(U_{k_j})\prec\mathcal V$. This implies that $\bigcap_{|j|\leq N}f^j(U^Y_{k_j})\prec\mathcal V^Y$ and the proof ends.
\end{proof}

\subsection{Refinement expansivity on $T_1$ spaces}\label{secRefImpOrb}

 \par\noindent\par We know that refinement expansivity does not imply orbit expansivity on arbitrary compact topological spaces, see Example \ref{exaNoT1}.

\begin{teo}\label{teoRefT1Cov}
  If $f\colon X\to X$ is refinement expansive and $X$ is a $T_1$ space, then $f$ is orbit expansive. Moreover, $X$ has a countable basis for its topology.
\end{teo}

\begin{proof}
  We start proving that $f$ is orbit expansive. Let $\mathcal U=\{U_1,\dots,U_n\}$ be an r-expansivity cover of $X$. Given any $(k_j)_{j\in\Z}\in\{1,\dots,n\}^{\Z}$ and two distinct points $x$ and $y$, consider the open cover $\mathcal V=\{X\backslash\{x\},X\backslash\{y\}\}$ (points are closed because $X$ is $T_1$) and $N\in\N$ such that $\mathcal U_N$ refines $\mathcal V$. Then $\{x,y\}\nsubseteq\bigcap_{|j|\leq N}f^j(U_{k_j})$, hence $\{x,y\}\nsubseteq\bigcap_{j\in\Z}f^j(U_{k_j})$. As $x,y$ are arbitrary, we conclude that $\card\bigl(\bigcap_{j\in\Z}f^j(U_{k_j})\bigr)\leq1$, and this is true for all $(k_j)_{j\in\Z}\in\{1,\dots,n\}^{\Z}$. Then $f$ is orbit expansive by Remark \ref{dfexpcard}.

 \par Now we show that $X$ has a countable basis. Let $\mathcal U=\{U_1,\dots,U_n\}$ be an r-expansive covering for $X$. Given a neighborhood $V$ of a point $x\in X$, consider the open cover $\mathcal V=\{V,X\backslash\{x\}\}$. Define a sequence $(k_j)\in\{1,\dots, n\}^{\Z}$ taking for each $j\in\Z$ a $k_j\in\{1,\dots, n\}$ such that $f^{-j}(x)\in U_{k_j}$. By refinement expansivity, for this sequence there is an $N\in\N$ such that $\bigcap_{|j|\leq N}f^j(U_{k_j})$ refines $\mathcal V$. As $x\in\bigcap_{|j|\leq N}f^j(U_{k_j})$ the later open set is not included in $X\backslash\{x\}$, so $x\in\bigcap_{|j|\leq N}f^j(U_{k_j})\subseteq V$. Therefore, the countable family $\bigcup_{N\in\N}\mathcal U_{f,N}$ is a basis for the topology.
\end{proof} 

 \par Finally, we show that a compact $T_1$ space with a countable basis admitting a refinement expansive homeomorphism may fail to be metrizable.

\begin{exa}\label{exadupli}
Let $(X,\tau)$ be a compact $T_1$ space and $f\colon X\to X$ a refinement expansive homeomorphism with a fixed point $x_0$. For example, we can take $f$ to be a metric expansive homeomorphism (as a symbolic shift map or an Anosov diffeomorphism of a compact manifold). Considering a new point $x_1\notin X$ we define $\bar X=X\cup \{x_1\}$, equipped with the topology $\bar\tau=\tau \cup\{ W\cup \{x_1\}:\ x_0\in W, \ W\in\tau\}\cup\{(W\setminus\{x_0\})\cup \{x_1\}: \ x_0\in W, \ W\in\tau\}$. It is easy to prove that $\bar X$ is $T_1$. Moreover, $\bar X$ is not Hausdorff if $x_0$ is not an isolated point of $X$.

 \par Let $g\colon \bar X\to\bar X$ be the function defined by $g(x)=f(x)$ if $x\in X$, $g(x_1)=x_1$. It is easy to see that $g$ is a homeomorphism. We will show that $g$ is refinement expansive. Since $f$ is refinement expansive, there exists $\mathcal U=\{U_1,\ldots, U_n\}$ an r-expansivity covering for $f$. Suppose $x_0\in U_n$ and define the open set $U_{n+1}=(U_n\setminus\{x_0\})\cup \{x_1\}$. We will prove that $\mathcal Z=\{U_1, \ldots,U_n,U_{n+1}\}$ is an r-expansivity covering for $g$.

 \par Suppose $\mathcal V=\{V_1,\ldots,V_r\}$ is an arbitrary covering of $\bar X$, with $x_0\in V_{r-1}$ and $x_1\in V_r$. For each $i=1,\ldots, r$ take $W_i=V_i\setminus\{x_0,x_1\}$ and define $W_{r+1}=(V_{r-1}\cap V_r)\cup\{x_0\}$. Then $\mathcal W=\{W_1,\ldots,W_r,W_{r+1}\}$ is a covering of $X$. Since $f$ is refinement expansive, there exists $N\in \N$ such that $\mathcal U_{f,N}\prec \mathcal W$. It only remains to prove that $\mathcal Z_{f,N}\prec \mathcal V$. To do this, observe at first that $\{x_0,x_1\}\not\subseteq g^j(U_k)$ for all $j\in\Z$ and $k=1,\ldots,n+1$ (recall that $x_0$ and $x_1$ are fixed points of $g$). Consider a sequence $(k_j)_{j\in\Z}\in\{1,\dots,n+1\}^{\Z}$. We first suppose that
 \begin{equation}\label{ecuNacho}
   x_1\notin\bigcap_{|j|\leq N}g^j(U_{k_j}).
 \end{equation}
 Notice that each time $k_j=n+1$ we can replace $k_j$ by $n$ in (\ref{ecuNacho}) and the intersection (\ref{ecuNacho}) will be the same. Therefore, 
   $$\bigcap_{|j|\leq N}g^j(U_{k_j})=\bigcap_{|j|\leq N}f^j(U_{k_j})\prec \mathcal W.$$ 
 This and (\ref{ecuNacho}) imply $\bigcap_{|j|\leq N}g^j(U_{k_j})\subseteq V_l$ for some $l<r$. Now assume that (\ref{ecuNacho}) does not hold. Then $k_j=n+1$ for all $j$ and 
   $$\bigcap_{|j|\leq N}g^j(U_{k_j})=\left[\bigcap_{|j|\leq N}f^j(U_{n})\setminus\{x_0\}\right]\cup\{x_1\}.$$
 In this case $\bigcap_{|j|\leq N}f^j(U_{n})$ must be contained in $W_{r+1}$ and $\bigcap_{|j|\leq N}g^j(U_{k_j})\subseteq V_r$. This proves that $g$ is refinement expansive.
\end{exa}

 \par We can generalize this example as follows. Let $(X,\tau)$ be a compact $T_1$ space, $f\colon X\to X$ a refinement expansive homeomorphism, and $K\subseteq X$ a closed invariant subset. We define $\bar X=X\cup \{(k,1): k\in K\}$ (we add a copy of $K$), equipped with the topology 
   $$\bar\tau=\tau \cup\{ W\cup ((W\cap K)\times\{1\}):\ W\in\tau\}\cup \{(W\setminus K)\cup ((W\cap K)\times\{1\}): \ W\in\tau\}.$$
 We define $f_K\colon \bar X\to\bar X$ as $f_K(x)=f(x)$ if $x\in X$, $f_K(z,1)=(f(z),1)$ if $z\in K$. It can be proven, as we did before, that $f_K$ is a refinament expansive homeomorphism.

 \par It would be interesting to know if every refinement expansive homeomorphism of a compact $T_1$ space is obtained in this way, starting with a metric expansive homeomorphism of a compact metric space.

\subsection{Further questions}

 \par\noindent\par Several questions remain open. For example, it is known that if a compact metric space $X$ admits a positive expansive homeomorphism then $X$ is a finite set. Is this result true for a positive refinement expansive homeomorphism on a $T_1$ space?
 
 \par In the metric case it is known that no one-dimensional manifold (the circle and the interval) admit an expansive homeomorphism. In the non-Hausdorff setting, there are other compact \emph{one-dimensional manifolds}. Consider for example two circles $C_1,C_2$ with two open subsets $U_1\subset C_1$, $U_2\subset C_2$ and a homeomorphism $h\colon U_1\to U_2$. Consider the space $X=(C_1\cup C_2)/x\sim h(x)$ for $x\in U_1$. Does such non-Hausdorff one-dimensional manifolds admit refinement expansive homeomorphisms?

\end{document}